\documentclass[a4paper]{article}
\usepackage{fullpage}
\usepackage{tikz}

\usepackage[pdftex,breaklinks,colorlinks,
linkcolor=blue,
citecolor=blue,
urlcolor=blue]{hyperref}

\usepackage{enumitem}
\usepackage{amsmath,amssymb,amsfonts,amsthm}

\newtheorem{theorem}{Theorem}

\newtheorem{lemma}[theorem]{Lemma}

\newtheorem{corollary}[theorem]{Corollary}

\theoremstyle{definition}
\newtheorem{definition}[theorem]{Definition}
\newtheorem{remark}[theorem]{Remark}

\author{Pablo Romero\footnote{Facultad de Ingenier\'ia, Universidad de la Rep\'ublica, Montevideo, Uruguay. E-mail address: \texttt{promero@fing.edu.uy}.\\ Departamento de Matem\'atica, Facultad de Ciencias Exactas y Naturales, Universidad de Buenos Aires, Argentina. 
}\qquad Louis Petingi\footnote{College of Staten Island, City University of New York, USA. E-mail address: \texttt{louis.petingi@csi.cuny.edu}}}

\date{}

\makeatletter%
\begin{document}

\title{Construction of infinitely many trace-minimal graphs with maximum number of spanning trees}

\maketitle

\begin{abstract}\let\thefootnote\relax
A longstanding problem in spectral graph theory asks for graphs with maximum number of spanning trees among all connected simple graphs with a prescribed number of vertices and edges. Such graphs are called $t$-optimal graphs. Petingi and Rodr\'iguez [Discrete Math. 244 (2002), 351--373] achieved in finding infinitely many $t$-optimal graphs. Basically, they reduced the problem of finding $t$-optimal graphs to the determination of almost-regular graphs with minimum number of induced $3$-paths. 

In this work we revisit the construction of $t$-optimal graphs given by Petingi and Rodríguez. Then, we 
generalize the previous construction using the key concept of trace-minimal graph introduced by \'Abrego et al. [Linear Algebra Appl. 412 (2006) 161--221]. Finally, as a consequence, we construct infinitely many new $t$-optimal regular graphs.
\end{abstract}

\renewcommand{\labelitemi}{--}
\section{Introduction} \label{intro}
In the half of the nineteenth century Kirchhoff, interested in the solution of linear resistive circuits, 
proved that the number of spanning trees of any graph equals each of the cofactors of its Laplacian matrix~\cite{Biggs}. 
As a consequence, the number of spanning trees of a graph, sometimes called the \emph{tree-number of a graph}, can be found efficiently. Nevertheless, if we are given integers $n$ and $m$ a question that arises is how to construct simple graphs with $n$ vertices and $m$ edges maximizing the tree-number. 
Such graphs, called \emph{$t$-optimal graphs}, are essential in network reliability analysis~\cite{Romero}.

Cheng~\cite{Cheng} proved that each complete multipartite graph is $t$-optimal. Later, Petingi and Rodríguez~\cite{Petingi} generalized Cheng theorem by proving that each almost-regular complete multipartite graph is $t$-optimal. They developed a methodology to construct $t$-optimal graphs which basically reduces the problem of finding $t$-optimal graphs to the determination of almost-regular graphs with minimum number of induced $3$-paths. As a consequence, they achieved in finding infinitely many $t$-optimal graphs on $n$ vertices and $m$ edges when $\binom{n}{2}-3n/2\leq m \leq \binom{n}{2}$. Additional $t$-optimal graphs were only determined for graph classes having reduced corank. In fact, Kahl and Luttrell~\cite{Kahl} introduced the concept of Tutte-maximum graphs and proved that each Tutte-maximum graph is not only $t$-optimal but also maximizes simultaneously several graph invariants. Then, they determined infinitely many Tutte-maximum graphs with reduced corank (the reader is invited to consult~\cite{Kahl} for details).

Surprisingly, an article that was published two decades ago whose goal was completely different gives a powerful ingredient to the study of $t$-optimal graphs. \'Abrego et al.~\cite{Abrego} wanted to find D-optimal weighting designs (basically, how to estimate the weight of some objects with minimum uncertainty). Using a novel concept of \emph{trace-minimal graphs} they managed to find new D-optimal weighting designs. 

The purpose of this article is to show that the concept of trace-minimal graphs can be used to enhance the methodology presented by Petingi and Rodríguez in~\cite{Petingi}. As a consequence we will find novel $t$-optimal regular graphs that are trace-minimal as well. The article is organized as follows. Section~\ref{section:concepts} presents general concepts on graph theory. Section~\ref{section:background} presents key concepts and statements from the works of Petingi and Rodríguez~\cite{Petingi} as well as \'Abrego et al.~\cite{Abrego}. An enhanced methodology to find $t$-optimal graphs based on those works is introduced in Section~\ref{section:main}. As a consequence, 
novel $t$-optimal graphs are given in Section~\ref{section:toptimal}. 

\section{Concepts}\label{section:concepts}
Let $\mathcal{S}_{n,m}$ be the class of simple graphs on $n$ vertices and $m$ edges. Let $G$ be any graph in $\mathcal{S}_{n,m}$. 
We denote its vertex and edge set by $V(G)$ and $E(G)$, respectively. Let $V(G)=\{v_1,v_2,\ldots,v_n\}$. 
Two vertices $v_i$ and $v_j$ in $G$ are \emph{adjacent} when 
$v_iv_j \in E(G)$; in such case the edge $v_iv_j$ is \emph{incident} at both vertices $v_i$ and $v_j$. The \emph{degree} of a vertex $v_i$ in $G$, denoted $d_i$, is the number of edges incident at $v_i$. 
The \emph{degree sequence} $d(G)$ of $G$ is $(d_1,d_2,\ldots,d_n)$. We say $G$ is \emph{$d$-regular} if $d_i=d$ for each $i\in \{1,2,\ldots,n\}$. Let $\mathcal{R}_d(n)$ be the class of $d$-regular graphs on $n$ vertices. We say $G$ is \emph{almost-regular} if the degrees of any two of its vertices differ in at most  $1$. The class of almost-regular graphs in $\mathcal{S}_{n,m}$ is denoted $\mathcal{A}_{n,m}$. The identity matrix and the all-ones matrix of size $n \times n$ are denoted $I_n$ and 
$J_n$, respectively. The all-ones vector of size $n$ is denoted $\mathbf{1}_n$. 
The \emph{trace} of a matrix $M$, denoted $tr(M)$, is the sum of the elements in its diagonal. The \emph{adjacency matrix of $G$} is the matrix $A(G)=(a_{i,j})_{1\leq i,j\leq n}$ such that $a_{i,j}=1$ if $v_iv_j \in E(G)$ or $a_{i,j}=0$ otherwise. Let $D(G)$ the diagonal matrix whose diagonal is $d(G)$. 
The \emph{Laplacian matrix of $G$} is the matrix $L(G)$ given by $D(G)-A(G)$. Let $P_G(x)$ be the characteristic polynomial of $L(G)$. The \emph{Laplacian spectrum of $G$} is the multiset consisting of all roots of $P_G(x)$. 

For each $S$ in $V(G)$, the \emph{subgraph of $G$ induced by $S$} arises from $G$ by the deletion of each vertex not in $S$. Let $t(G)$, $\tau(G)$, and $\nu(G)$ be the number of spanning trees, induced triangles, and induced $3$-paths in $G$, respectively. The graph $G$ is \emph{$t$-optimal} if $t(G)\geq t(H)$ for each $H$ in $\mathcal{S}_{n,m}$. 
The \emph{girth} of $G$ is the minimum number of vertices in a cycle of $G$ (it is infinite if $G$ has no cycles). For each positive integer $i$, we denote $cyc(G,i)$ the number of cycles in $G$ having precisely $i$ vertices. If $G$ is in $\mathcal{A}_{n,m}$, then it is \emph{$\nu$-min} if $\nu(G)\leq \nu(H)$ for each $H$ in $\mathcal{A}_{n,m}$. The \emph{complement of $G$} is denoted $\overline{G}$. The union of two disjoint graphs $G$ and $H$ is denoted $G \cup H$. The \emph{join} of two graphs $G$ and $H$, denoted $G \wedge H$, is the graph $\overline{\overline{G} \cup \overline{H}}$. We denote $G^{(n)}$ the join of $n$ disjoint copies of $G$. The number of vertices of $G$ is denoted $n(G)$. Denote $\mathcal{S}(G)$, $\mathcal{A}(G)$, and 
$\mathcal{R}(G)$ the classes of simple, almost-regular, or regular graphs with precisely as many vertices and edges as $G$, respectively. The $n$-path, the $n$-cycle, and the $n$-complete graph are denoted by $P_n$, $C_n$, and $K_n$, respectively. 

\section{Related work}\label{section:background}
In this section we will first revisit the methodology developed by Petingi and Rodríguez~\cite{Petingi} to construct $t$-optimal graphs. Then, we will present the concept of trace-minimal graphs introduced by \'Abrego et al.~\cite{Abrego}. Finally, we will list some results on trace-minimal graphs which will be useful for our purpose.\\

The number of spanning trees of a graph $G$ is determined by its Laplacian spectrum. 
\begin{lemma}[Biggs~\cite{Biggs}]\label{lemma:t}
If $G$ is a simple graph on $n$ vertices then $t(G)=n^{-2}P_{\overline{G}}(n)$.    
\end{lemma}
Let $G$ be any simple graph on $n$ vertices. Denote $P_G(x)=\prod_{i=1}^{n}(x-\lambda_i)$, where $\lambda_1,\lambda_2,\ldots,\lambda_n$ are the Laplacian eigenvalues of $G$. 
If $\overline{G}$ is connected then it is simple to prove that each of the Laplacian eigenvalues of $G$ lie in the interval $[0,n)$ and the function $P_G(x)/x^n$ is positive for all $x\geq n$. Additionally, 
\begin{align}
-\log\left(\frac{P_G(x)}{x^n}\right) &= -\log\left(\prod_{i=1}^{n}(1-\frac{\lambda_i}{x})\right) = \sum_{i=1}^{n}-\log\left(1-\frac{\lambda_i}{x}\right)  = \sum_{i=1}^{n}\sum_{k=1}^{\infty}\frac{\lambda_i^k}{kx^k} \notag\\ 
&= \sum_{k=1}^{\infty}\frac{\sum_{i=1}^n \lambda_i^k}{kx^k} = \sum_{k=1}^{\infty}\frac{tr(L(G)^k)}{kx^k} = \sum_{k=1}^{\infty}\frac{\ell_k(G)}{kx^k}, \label{eq:nodrop}
\end{align}
where the power series expansion $-\log(1-x)=\sum_{k=1}^{\infty}\frac{x^k}{k}$ was used which is valid whenever $|x|<1$, and $\ell_k(G)=tr(L(G)^k)$ is the \emph{Laplacian sequence} of $G$. Petingi and Rodríguez found lower bounds for $\ell_k(G)$. 
\begin{lemma}[Petingi and Rodríguez~\cite{Petingi}]\label{lemma:gaps}
If $G$ is a simple graph with degree sequence $(d_1,d_2,\ldots,d_n)$ and $k$ is any positive integer then  
$\ell_k(G) \geq \sum_{i=1}^{n}d_i(d_i+1)^{k-1}$. Additionally, the equality holds for each positive integer $k$ if and only if $G$ is a disjoint union of complete graphs.
\end{lemma}
The \emph{gap sequence} of a graph $G$ with degree sequence $(d_1,d_2,\ldots,d_n)$ 
is $g_k(G)=\ell_k(G)-\sum_{i=1}^{n}d_i(d_i+1)^{k-1}$. Petingi and Rodríguez observed that $g_1(G)=g_2(G)=0$ and $g_3(G)=2\nu(G)$. Consequently, for all $x\geq n$,
\begin{align}\label{ineq:chain}
-\log\left(\frac{P_G(x)}{x^n}\right) &= \sum_{k=1}^{\infty}\frac{\ell_k(G)}{kx^k} \geq 
\sum_{k=1}^{\infty}\frac{\sum_{i=1}^{n}d_i(d_i+1)^{k-1}}{kx^k} + \frac{2\nu(G)}{3x^3}= \sum_{i=1}^{n}\frac{d_i}{d_i+1}\sum_{k=1}^{\infty}\frac{(d_i+1)^k}{kx^k} + \frac{2\nu(G)}{3x^3}\\
&= -\sum_{i=1}^{n}\frac{d_i}{d_i+1}\log\left(1- \frac{d_i+1}{x}\right)+ \frac{2\nu(G)}{3x^3} = -\log\left(e^{-\frac{2\nu(G)}{3x^3}}\prod_{i=1}^{n}\left(1-\frac{d_i+1}{x} \right)^{\frac{d_i}{d_i+1}} \right). \notag
\end{align}
Solving for $P_G(x)$ yields
\begin{equation}\label{eq:poly}
P_G(x) \leq x^n e^{-\frac{2\nu(G)}{3x^3}}\prod_{i=1}^{n}\left( 1- \frac{d_i+1}{x}\right)^{\frac{d_i}{d_i+1}}.    
\end{equation}
The authors obtained the following result replacing~\eqref{eq:poly} into the expression for $t(\overline{G})$ given in Lemma~\ref{lemma:t}. 
\begin{lemma}[Petingi and Rodríguez~\cite{Petingi}]\label{lemma:PetingiRodriguez}
If $G$ is a graph with degree sequence $d_1,\ldots,d_n$ such that $\overline{G}$ is connected then  
\begin{equation}\label{eq:bound1}
t(\overline{G}) \leq n^{n-2}e^{-\frac{2\nu(G)}{3n^3}}\prod_{i=1}^{n}\left( 1- \frac{d_i+1}{n}\right)^{\frac{d_i}{d_i+1}}.
\end{equation}
The equality occurs if and only if $G$ is a union of complete graphs.
\end{lemma}
Define the function $f(d,x)$ as the product operator that appears on the right-hand side of equation~\eqref{eq:poly},
\begin{equation}\label{eq:f}
f(d,x) = \prod_{i=1}^{n}\left( 1- \frac{d_i+1}{x}\right)^{\frac{d_i}{d_i+1}}.   
\end{equation}
The authors proved that $f(d,x)$ is maximized among simple graphs $G$ in $\mathcal{S}_{n,m}$ when $G$ is almost-regular. 
\begin{lemma}[Petingi and Rodríguez~\cite{Petingi}]\label{lemma:f}
If $G$ is in $\mathcal{S}_{n,m}-\mathcal{A}_{n,m}$ and $H$ is in 
$\mathcal{A}_{n,m}$ then $f(d(G),x)< f(d(H),x)$ for all $x\geq n$.
\end{lemma}
A longstanding conjecture proposed by Boesch~\cite{Boesch} states that each $t$-optimal graph is almost-regular. Petingi and Rodríguez proved that Boesch conjecture holds in an asymptotic sense which is precisely stated in Theorem~\ref{theorem:PetingiRodriguez-almostregular}. The following notation will be used throughout this article. For each almost-regular graph $G_0$ all of whose vertices have degree $d-1$ or $d$ and each pair of nonnegative integers $p$ and $q$, we let 
$G_0(p,q)$ be the graph $G_0 \cup pK_{d+1} \cup qK_d$. Note that  $G_0(p,q)$ and $\overline{G_0(p,q)}$ are almost-regular.
\begin{theorem}[Petingi and Rodríguez~\cite{Petingi}]\label{theorem:PetingiRodriguez-almostregular}
Let $G_0$ be an almost-regular graph all of whose vertices have degree $d-1$ or $d$. Then there exists a positive integer $n_0$ such that each graph $H$ in $\mathcal{S}(G_0(p,q))$ with at least $n_0$ vertices that is not almost-regular satisfies that $t(\overline{H})<t(\overline{G_0(p,q)})$.
\end{theorem}

\begin{remark}\label{remark:cobertura}
The authors in~\cite{Petingi} found values for $n_0$ satisfying the conditions of Theorem~\ref{theorem:PetingiRodriguez-almostregular}. A possible choice for $n_0$ is $2d+n(G_0)(2d)^3$.  
\end{remark}

A refinement of Theorem~\ref{theorem:PetingiRodriguez-almostregular} takes into consideration 
the fact that, among almost-regular graphs, Lemma~\ref{lemma:PetingiRodriguez} gives priority to $\nu$-min graphs. 
\begin{theorem}[Petingi and Rodríguez~\cite{Petingi}]\label{theorem:PetingiRodriguez}
Let $G_0$ be an almost-regular graph all of whose vertices have degree $d-1$ or $d$. Suppose that $G_0(p,q)$ is $\nu$-min for all $p$ and $q$. Then there exists a positive integer $n_0$ such that each graph $H$ in $\mathcal{S}(G_0(p,q))$ that is not $\nu$-min satisfies that $t(\overline{H})<t(\overline{G_0(p,q)})$ whenever $n(H)\geq n_0$.
\end{theorem}
The authors in~\cite{Petingi} also found specific values for $n_0$ satisfying the conditions of Theorem~\ref{theorem:PetingiRodriguez}. The methodology developed by Petingi and Rodríguez basically consists in finding an almost-regular graph $G_0$ meeting the conditions of Theorem~\ref{theorem:PetingiRodriguez} and then find, among 
each $\nu$-min graph in $\mathcal{A}(G_0(p,q))$, the graph (or graphs) whose complement has the maximum number of spanning trees.
\begin{remark}\label{remark:key}
In the inequality of expression~\eqref{ineq:chain}, the authors replaced $\ell_k(G)$ by $\sum_{i=1}^{n}d_i(d_i+1)^{k-1}$ (except when $k=3$). 
However, as $\ell_k(G)$ equals $g_k(G)+\sum_{i=1}^{n}d_i(d_i+1)^{k-1}$,  sharper bounds for $t(\overline{G})$ could be obtained.
\end{remark}

We will employ Remark~\ref{remark:key} to find sharper bounds for the number of spanning trees. As a consequence, a generalization of Theorem~\ref{theorem:PetingiRodriguez} will be given (see Theorem~\ref{theorem:generalize}). A key concept to achieve this goal is that of a trace-minimal graph~\cite{Abrego}; a lexicographic order among sequences of real numbers is first required.
\begin{definition}
Given two sequences of real numbers $(b_i)_{i\in \mathbb{Z}^+}$ and $(c_i)_{i\in \mathbb{Z}^+}$, 
we write  $b_i \preceq c_i$ when precisely one of the following conditions holds:
\begin{enumerate}[label=(\roman*)]
\item For each positive integer $i$ it holds that $b_i = c_i$, or 
\item There exists a positive integer $j$ such that for each $i\in \{1,2,\ldots,j-1\}$ it holds that $b_i\leq c_i$ but $b_j<c_j$.
\end{enumerate}
\end{definition}

For each $G$ in $\mathcal{S}_{n,m}$ we define its \emph{adjacency sequence} $(a_i)_{i\in \mathbb{Z}^+}$ as $a_i(G)=tr(A(G)^i)$.  

\begin{definition}[\'Abrego et al.~\cite{Abrego}]
A graph $G$ in $\mathcal{R}_{d}(n)$ is \emph{trace-minimal} if for each $H$ in $\mathcal{R}_{d}(n)$, $a_i(G) \preceq a_i(H)$. 
\end{definition}
To close this section, we will give a list of results concerning trace-minimal graphs that appeared in~\cite{Abrego}. 

\begin{theorem}[\'Abrego et al.~\cite{Abrego}]
Let $G$ be a graph with maximum girth $g$ in $\mathcal{R}_d(n)$. 
Suppose that for each $H$ in $\mathcal{R}_d(n)$ there exists an integer $k$ such that $k\leq 2g-1$, 
$cyc(G,i)=cyc(H,i)$ for each $i\in \{3,4,\ldots,k-1\}$, and $cyc(G,k)<cyc(H,k)$. Then, $G$ is trace-minimal in $\mathcal{R}_d(n)$. 
\end{theorem}

\begin{corollary}[\'Abrego et al.~\cite{Abrego}]
If $G$ is the only graph in $\mathcal{R}_d(n)$ with maximum girth then $G$ is the only trace-minimal graph in $\mathcal{R}_d(n)$. 
\end{corollary}

Recall that $\tau(G)$ denotes the number of triangles in $G$. Let $\tau_d(n)=\min\{\tau(G): G\in \mathcal{R}_d(n)\}$.
\begin{lemma}[\'Abrego et al.~\cite{Abrego}]\label{lemma:test}
Let $G$ be a graph in $\mathcal{R}_{n-\delta-1}(n)$ for some integer $\delta$ such that $\delta\geq 3$. 
Let $q$ and $\rho$ be nonnegative integers satisfying $n=q(\delta+1)+\rho$ where $\rho\in \{0,1,\ldots,\delta\}$. 
If $G$ is trace-minimal, then either $G=H\wedge \overline{K_{\delta+1}}$ for some trace-minimal graph $H$ in $\mathcal{R}_{n-2\delta-2}(n-\delta-1)$ or the following inequality holds
\begin{equation*}
n \leq \frac{\rho}{4}((\delta+1)^2-\rho^2) + \frac{3}{2}\tau_{\rho}(d+1+\rho).    
\end{equation*}
\end{lemma}

\section{Main results}\label{section:main}
The main result of this section is Theorem~\ref{theorem:summary}, which states that there exists infinitely many trace-minimal graphs that are $t$-optimal. First, let us find bounds for $t(\overline{G})$ 
that are sharper than the one given in Lemma~\ref{lemma:PetingiRodriguez} by using Remark~\ref{remark:key}.

\begin{lemma}\label{lemma:improvedbound}
For each graph $G$ on $n$ vertices such that $\overline{G}$ is connected and each positive integer $c$,
$P_G(x) \leq x^n e^{-\sum_{k=1}^{c}\frac{g_k(G)}{kx^k}} f(d(G),x)$,  
for all $x\geq n$. Further, $t(\overline{G}) \leq n^{n-2}e^{-\sum_{k=1}^{c}\frac{g_k(G)}{kx^k}}f(d(G),n)$, where $f$ is defined in expression~\eqref{eq:f}. 
\end{lemma}
\begin{proof}
Let $c$ be any positive integer. By expression~\eqref{eq:nodrop} and Remark~\ref{remark:key}, for all $x\geq n$, 
\begin{align*}
-\log\left(\frac{P_G(x)}{x^n}\right) &=\sum_{k=1}^{\infty}\frac{\ell_k(G)}{kx^k} \geq 
\sum_{k=1}^{\infty}\frac{\sum_{i=1}^{n}d_i(d_i+1)^{k-1}}{kx^k} + \sum_{k=1}^{c}\frac{g_k(G)}{kx^k}\\ 
&= -\log\left(e^{-\sum_{k=1}^{c}\frac{g_k(G)}{kx^k}}\prod_{i=1}^{n}\left(1-\frac{d_i+1}{x} \right)^{\frac{d_i}{d_i+1}} \right),
\end{align*}
where the last equality follows from~\eqref{ineq:chain}. 
Solving for $P_G(x)$ yields the first part of the statement. The second part of the statement follows from Lemma~\ref{lemma:t}.
\end{proof}

Define $\mathcal{S}_{n,m}^{(1)}$ as $\mathcal{S}_{n,m}$. For each positive integer $k$, define $\mathcal{S}_{n,m}^{(k+1)}$ 
as $\mathcal{S}_{n,m}^{(k+1)} = \{G: G\in \mathcal{S}_{n,m}^{(k)}, \ell_{k+1}(G)\leq \ell_{k+1}(H) \text{ for each } H \text{ in }  \mathcal{S}_{n,m}^{(k)}\}$.    
It is simple to check that $\mathcal{S}_{n,m}^{(2)}=\mathcal{A}_{n,m}$ and $\mathcal{S}_{n,m}^{(3)}$ is the set consisting of all $\nu$-min graphs in $\mathcal{A}_{n,m}$. 
In Theorem~\ref{theorem:generalize} we will generalize Theorem~\ref{theorem:PetingiRodriguez}. First, let us prove two technical lemmas.
\begin{lemma}\label{lemma:gmh}
Let $G_0$ be an almost-regular graph all of whose vertices have degree $d-1$ or $d$ and whose Laplacian eigenvalues are 
$\lambda_1,\lambda_2,\ldots,\lambda_n$. Let $p$ and $q$ be nonnegative integers. Define the function $g(x)$ as follows,
\begin{equation}\label{eq:check}
g(x) = \left(1-\frac{d+1}{x}\right)^{pd}\left( 1-\frac{d}{x}\right)^{q(d-1)}\prod_{i=1}^{n}\left(1-\frac{\lambda_i}{x}\right).    
\end{equation}    
Let $n'=n(G_0(p,q))$. Then, the following assertions hold.
\begin{enumerate}[label=(\roman*)]
\item\label{gmh2} $t(\overline{G_0(p,q)}) = (n')^{n'-2}g(n')$.
\item\label{gmh3} The gap sequences of $G_0(p,q)$ and $G_0$ are identical.
\end{enumerate}
\end{lemma}
\begin{proof}
As $J_n$ has rank $1$ and $J_n\mathbf{1}_n = n\mathbf{1}_n$, it has $n-1$ eigenvalues equal to $0$ and a single eigenvalue equal to $n$. As $L(K_n)$ equals $nI_n-J_n$, it has $n-1$ eigenvalues equal to $n$ and a single eigenvalue equal to $0$ and $P_{K_n}(x)=x(x-n)^{n-1}$. As the Laplacian polynomial factors over components, 
$P_{G_0(p,q)}(x)=x^{p+q}(x-(d+1))^{pd}(x-d)^{q(d-1)}\prod_{i=1}^{n}(x-\lambda_i)$. Then, $P_{G_0(p,q)}(x)=x^{n'}g(x)$ and  Lemma~\ref{lemma:t} gives that $t(\overline{G_0(p,q)})=(n')^{-2}P_{G_0(p,q)}(n')=(n')^{n'-2}g(n')$ thus proving~\ref{gmh2}. Finally, let $k$ be any positive integer. As $g_k$ is additive over disjoint graphs, $g_k(G_0(p,q))=g_k(G_0)+pg_k(K_{d+1})+qg_k(K_d)= g_k(G_0)$,  
where we used the second part of the statement of Lemma~\ref{lemma:gaps} for the last equality, thus proving~\ref{gmh3}.
\end{proof}

\begin{lemma}\label{lemma:boundgap}
Let $G_0$ be in $\mathcal{A}_{n,m}$ all of whose vertices have degree $d-1$ or $d$. 
For each positive integer $k$, $g_k(G_0)\leq n (2d)^{k}$.
\end{lemma}
\begin{proof}
Let $k$ and $G_0$ be as in the statement. Let $\lambda_1,\lambda_2,\ldots,\lambda_n$ be the Laplacian eigenvalues of $G_0$, and let $i\in \{1,2,\ldots,n\}$. As $L(G_0)$ is a semidefinite positive symmetric matrix, we know that $\lambda_i\geq 0$. Additionally, by Gershgorin theorem, $\lambda_i$ lies in the interval $[0,2d]$. On the one hand, $\ell_k(G_0)=\sum_{i=1}^{n}\lambda_i^k\leq n(2d)^k$. On the other hand, 
$\sum_{i=1}^{n}d_i(d_i+1)^{k-1}\geq 0$. Consequently, 
$g_k(G_0)=\ell_k(G_0)-\sum_{i=1}^{n}d_i(d_i+1)^{k-1} \leq n(2d)^k$, 
and the lemma follows.
\end{proof}

\begin{theorem}\label{theorem:generalize}
Let $c$ be any positive integer. Let $G_0$ be an almost-regular graph all of whose vertices have degree $d-1$ or $d$. 
Assume that for each pair of nonnegative integers $p$ and $q$ it holds that $G_0(p,q)\in \mathcal{S}(G_0(p,q))^{(c+1)}$. Then, for any graph $H$ in $\mathcal{S}(G_0(p,q))$ not in $\mathcal{S}(G_0(p,q))^{(c+1)}$ with at least $2d+n(G_0)(2d)^{c+2}$ vertices it holds that 
$t(\overline{H})<t(\overline{G_0(p,q)})$.
\end{theorem}
\begin{proof}
Let $c$ and $G_0$ be as in the statement. If $c=1$ then the statement follows from Theorem~\ref{theorem:PetingiRodriguez-almostregular} and Remark~\ref{remark:cobertura}. Let $c\geq 2$. 
Let $r$ and $s$ be the number of vertices in $G_0$ with degrees $d-1$ and $d$, respectively. Define $h(x,y)$  as follows,
\begin{equation*}
h(x,y) = e^{-\sum_{k=1}^c\frac{g_k(G_0)}{kx^k}-\frac{y}{(c+1)x^{c+1}}}\left(1-\frac{d+1}{x}\right)^{(s+p(d+1))d/(d+1)}\left(1-\frac{d}{x} \right)^{(r+qd)(d-1)/d}.
\end{equation*}
We claim that for each integer $g'>g_{c+1}(G_0)$ it holds that 
$h(x,g')<g(x)$ when $x\geq 2d+n(G_0)(2d)^{c+2}$. 

Once we prove our claim the theorem will follow. In fact, let $H$ be any graph in $\mathcal{S}(G_0(p,q)$ not in $\mathcal{S}(G_0(p,q))^{(c+1)}$. Observe that, without loss of generality, we can assume that $H$ is in $\mathcal{S}(G_0(p,q))^{(c)}$. 
Let $k$ be any positive integer. By Lemma~\ref{lemma:gmh}\ref{gmh3},  $g_k(G_0(p,q))=g_k(G_0)$. As $c\geq 2$, the graph $H$ is almost-regular thus $d(H)=d(G_0(p,q))$ and $g_k(H)-g_k(G_0(p,q))=\ell_k(H)-\ell_k(G_0(p,q))=\ell_k(H)-\ell_k(G_0)$. If $k\in \{1,2,\ldots,c\}$ 
then $g_k(H)=g_k(G_0)$. Let $n'=n(G_0(p,q))$ and $g'=g_{c+1}(H)$. By assumption, $g'>g_{c+1}(G_0)$. If $\overline{H}$ is not connected then $t(\overline{H})=0$ and 
 the result is immediate. Otherwise, Lemma~\ref{lemma:improvedbound} gives that $t(\overline{H})\leq (n')^{n'-2}h(n',g')$. 
By Lemma~\ref{lemma:gmh}\ref{gmh2} we know that $t(\overline{G_0(p,q)})=(n')^{n'-2}g(n')$. Therefore, if $p$ and $q$ are such that $n'\geq 2d+n(G_0)(2d)^{c+2}$ then
$t(\overline{H}) \leq (n')^{n'-2}h(n',g') <(n')^{n'-2}g(n')=t(\overline{G_0(p,q)})$.  

To prove our claim, let us consider the function $w(x)=g(x)/h(x,g')$, and its corresponding logarithm i.e.,
\begin{align*}
w(x) &= \prod_{i=1}^{n}\left(1-\frac{\lambda_i}{x}\right)\left(1-\frac{d+1}{x} \right)^{-sd/(d+1)}\left(1-\frac{d}{x} \right)^{-r(d-1)/d}e^{\sum_{k=1}^c\frac{g_k(G_0)}{kx^k}+\frac{g'}{(c+1)x^{c+1}}} \notag\\
 \log(w(x)) &= \sum_{i=1}^{n}\log\left(1-\frac{\lambda_i}{x}\right)-\frac{sd}{d+1}\log\left(1-\frac{d+1}{x}\right)-\frac{r(d-1)}{d}\log\left(1-\frac{d}{x}\right)+\sum_{k=1}^{c}\frac{g_k(G_0)}{kx^k}+\frac{g'}{(c+1)x^{p+1}}.   
\end{align*}
Now, when $x>2d$ we can take the derivative of the function $\log(w(x))$ with respect to $x$ and
\begin{align*}
\log(w(x))' &= \sum_{i=1}^{n}\frac{\lambda_i}{x^2(1-\lambda_i/x)}-\frac{sd}{x^2}\frac{1}{1-\frac{d+1}{x}}-\frac{r(d-1)}{x^2}\frac{1}{1-\frac{d}{x}}-\sum_{i=1}^{c}\frac{g_i(G_0)}{x^{i+1}}-\frac{g'}{x^{c+2}}\\
&= \frac{1}{x^2} \left(\sum_{j=0}^{\infty}\frac{(\sum_{i=1}^{n}\lambda_i^{j+1})-sd(d+1)^j-r(d-1)d^j}{x^j}   -\sum_{j=1}^{c}\frac{g_j(G_0)}{x^{j-1}}-\frac{g'}{x^c}\right)\\
&=\frac{1}{x^2} \left(\sum_{j=0}^{\infty}\frac{\ell_{j+1}(G_0)-\sum_{i=1}^{n}d_i(G_0)(d_i(G_0)+1)^{j}}{x^j} -\sum_{j=0}^{c-1}\frac{g_{j+1}(G_0)}{x^{j}}-\frac{g'}{x^c}\right)\\
&= \sum_{j=0}^{\infty}\frac{g_{j+1}(G_0)}{x^{j+2}}- \sum_{j=0}^{c-1}\frac{g_{j+1}(G_0)}{x^{j+2}}-\frac{g'}{x^{c+2}}=\frac{g_{c+1}(G_0)-g'}{x^{c+2}}+\sum_{k=c+3}^{\infty}\frac{g_{k-1}(G_0)}{x^k}\\
&\leq  \frac{g_{c+1}(G_0)-g'}{x^{c+2}}+\sum_{k=c+3}^{\infty}\frac{n(G_0)(2d)^{k-1}}{x^k} = \frac{g_{c+1}(G_0)-g'}{x^{c+2}} + \frac{n(G_0)}{2d}\frac{(2d)^{c+3}}{x^{c+3}}\frac{1}{1-(2dx^{-1})}\\ 
&\leq \frac{1}{x^{c+2}}\left(\frac{n(G_0)(2d)^{c+2}}{x-2d}-1\right),
\end{align*}
where the first inequality uses Lemma~\ref{lemma:boundgap} and the second inequality uses 
that $g'-g_{c+1}(G_0)\geq 1$ since both $g'$ and $g_{c+1}(G_0)$ are integers and 
$g'>g_{c+1}(G_0)$. Observe that  
$\log(w(x))'$ is negative when $x\geq 2d+n(G_0)(2d)^{c+2}$. As $w(x)$ approaches $1$ when $x$ tends to infinity we know that $w(x)>1$ when $x\geq 2d+n(G_0)(2d)^{c+2}$, or equivalently, the inequality $h(x,g')<g(x)$ holds when $x\geq 2d+n(G_0)(2d)^{c+2}$. The claim was proved, and the theorem follows.
\end{proof}

On the one hand, Theorem~\ref{theorem:PetingiRodriguez} basically shows that the complement of each $t$-optimal graph must be $\nu$-min when the number of  vertices is sufficiently large. 
On the other hand, Theorem~\ref{theorem:generalize} shows that the complement of such $t$-optimal graphs must be not only $\nu$-min, 
but also minimize the Laplacian sequence in the lexicographic order. The following concept is analogous to that of trace-minimal graphs and its motivation is Theorem~\ref{theorem:generalize}.
\begin{definition}
A graph $G$ in $\mathcal{R}_{d}(n)$ is \emph{$\mathcal{L}$-trace-minimal} if for each $H$ in $\mathcal{R}_{d}(n)$, $\ell_i(G) \preceq \ell_i(H)$. 
\end{definition}

For each graph $G_0$ in $\mathcal{R}_d(n)$ we let $G_0(p)=G_0\cup pK_{d+1}$. The following result is a corollary of Theorem~\ref{theorem:generalize}. 
\begin{corollary}\label{corollary:t-unique}
Let $G_0$ be in $\mathcal{R}_d(n)$. Assume that $G_0(p)$ is the only $\mathcal{L}$-trace-minimal graph in $\mathcal{R}(G_0(p))$ for each nonnegative integer $p$. 
Then, there exists $n_0$ such that $\overline{G_0(p)}$ is the only $t$-optimal graph in $\mathcal{S}(\overline{G_0(p)})$ when $n(G_0(p))\geq n_0$.  
\end{corollary}

We close this section showing a strong relationship between $\mathcal{L}$-trace-minimal and trace-minimal graphs. 
In fact, we will prove that a graph $G$ in $\mathcal{R}_d(n)$ is trace-minimal if and only if $\overline{G}$ is $\mathcal{L}$-trace-minimal in $\mathcal{R}_{n-1-d}(n)$. As a consequence, all the mathematical properties of trace-minimal graphs will be useful to find new $t$-optimal regular graphs. First, Lemma~\ref{lemma:basics} presents a basic result on linear algebra. 
\begin{lemma}\label{lemma:basics}
For each $G$ in $\mathcal{R}_d(n)$ and positive integers $i$ and $j$ it holds that $tr(L(G)^i((d+1)I_n-J_n))^j)=(d+1)^j\ell_i(G)$.
\end{lemma}
\begin{proof}
Let $i$ and $j$ be any positive integers. Observe that $J_n^2=nJ_n$. By induction it follows that $J_n^j=n^{j-1}J_n$. If we prove that $tr(L(G)^iJ_n)=0$ then the lemma follows since, by the linearity of the trace operator, 
\begin{align*}
tr(L(G)^i((d+1)I_n-J_n)^j) &= 
tr(L(G)^i \sum_{k=0}^{j}\binom{j}{k}((d+1)I_n)^k(-J_n)^{j-k})\\ 
&= \sum_{k=0}^{j}(-1)^{j-k}\binom{j}{k}(d+1)^ktr(L(G)^iJ_n^{j-k}) 
= (d+1)^j\ell_i(G).
\end{align*}
For each nonnegative integer $k$,  $tr(A(G)^kJ_n)$ equals the sum of all entries of $A(G)^k$, which is precisely the number of 
walks of length $k$ in $G$.  
As $G$ is a $d$-regular graph on $n$ vertices, $tr(A(G)^kJ_n)=nd^k$. Finally, as $L(G)=dI_n-A(G)$,
\begin{align*}
tr(L(G)^iJ_n) &= tr((dI_n-A(G))^iJ_n)=\sum_{k=0}^{i}\binom{i}{k}d^{i-k}(-1)^{k}tr(A(G)^{k}J_n)\\
&=\sum_{k=0}^{i}\binom{i}{k}d^{i-k}(-1)^{k}nd^{k}=nd^i\sum_{k=0}^{i}\binom{i}{k}(-1)^{k}=0 \qedhere
\end{align*}
\end{proof}

\begin{lemma}\label{lemma:lminimal}
A graph $G$ is $\mathcal{L}$-trace-minimal in $\mathcal{R}_d(n)$ if and only if $\overline{G}$ is trace-minimal in $\mathcal{R}_{n-1-d}(n)$.    
\end{lemma}
\begin{proof}
We will prove the direct since a proof of the converse is analogous. Let $G$ be any $\mathcal{L}$-trace-minimal graph in $\mathcal{R}_d(n)$. As $G$ is $\mathcal{L}$-trace-minimal, for each $H$ in $\mathcal{R}_d(n)$ one of the two conditions holds:
\begin{enumerate}[label=(\roman*)]
\item\label{c1} For each positive integer $i$ it holds that $\ell_i(G)=\ell_i(H)$, or 
\item\label{c2} There exists a positive integer $j$ such that $\ell_i(G)=\ell_i(H)$ for each $i\in \{1,2,\ldots,j-1\}$ 
but $\ell_j(G)<\ell_j(H)$.
\end{enumerate}
As $\overline{H}$ is any graph in $\mathcal{R}_{n-1-d}(n)$, it is enough to prove that $a_i(\overline{G})\preceq a_i(\overline{H})$. By Lemma~\ref{lemma:basics}, 
\begin{align*}
a_i(\overline{G})=tr(A(\overline{G})^i) = tr((L(G)-((d+1)I_n-J_n))^i) =
\sum_{k=0}^{i}(-1)^{i-k}\binom{i}{k}(d+1)^{i-k}\ell_k(G);\\
a_i(\overline{H})=tr(A(\overline{H})^i) = tr(L(H)-((d+1)I_n-J_n))^i) = 
\sum_{k=0}^{i}(-1)^{i-k}\binom{i}{k}(d+1)^{i-k}\ell_k(H).
\end{align*}
If condition~\ref{c1} holds then $a_i(\overline{G})=a_i(\overline{H})$ for each positive integer $i$ hence $a_i(\overline{G}) \preceq a_i(\overline{H})$. Finally, if condition~\ref{c2} holds then $a_i(\overline{G})=a_i(\overline{H})$ when $i\in \{1,2,\ldots,j-1\}$. However, each term in the summations for $a_j(\overline{G})$ and $a_j(\overline{H})$ is equal except for the last term which are respectively $\ell_j(G)$ and $\ell_j(H)$. As $\ell_j(G)<\ell_j(H)$, it follows that $a_j(\overline{G})<a_j(\overline{H})$.
\end{proof}

Theorem~\ref{theorem:summary} has a simple statement that summarizes our findings. 
\begin{theorem}\label{theorem:summary}
Let $G_0$ be in $\mathcal{R}_d(n)$. If $\overline{G_0(p)}$ is the only 
trace-minimal graph for each positive integer $p$, then there exists a positive integer $n_0$ such that $\overline{G_0(p)}$ is the only $t$-optimal whenever $n(G_0(p))\geq n_0$. If, in addition, each graph $H$ in $\mathcal{S}(G_0(p))-\{G_0(p)\}$ is not in $\mathcal{S}(G_0(p))^{c+1}$, then we can choose $n_0$ as $2d+n(G_0)(2d)^{c+2}$.
\end{theorem}

\section{Construction of new regular t-optimal graphs}\label{section:toptimal}
A generous list of trace-minimal graphs can be found in~\cite{Abrego}. Among them, we can find all trace-minimal graphs in each nonempty class $\mathcal{R}_{n-1-\delta}(n)$ where $\delta\in \{0,1,\ldots,5\}$. 
When $\delta\in \{0,1,2,3\}$, most trace-minimal graphs are $t$-optimal~\cite{Petingi}. Lemma~\ref{lemma:quartic} was proved by \'Abrego et al.~\cite{Abrego}. Its proof is an application of Lemma~\ref{lemma:test} with $\delta=4$.
\begin{lemma}[\'Abrego~\cite{Abrego}]\label{lemma:quartic}
Let $n$ be any integer such that $n\geq 5$, and let $q$ and $\rho$ be the only integers such that $\rho\in \{0,1,\ldots,4\}$ and $n=5q+\rho$. Then there exists precisely one trace-minimal graph $H_n$ in $\mathcal{R}_{n-5}(n)$. Such trace-minimal graph $H_n$ is $H_{5+\rho} \wedge \overline{K_5}^{(q-1)}$, 
where $H_5=\overline{K_5}$, $H_6=3K_2$, $H_7=C_7$, and $H_8$ and $H_9$ are depicted in Figure~\ref{figure:WG}.
\end{lemma}

\begin{figure}[!ht]
\begin{center}
\scalebox{0.9}{
\begin{tabular}{c}
\begin{tikzpicture}[scale=1]

\coordinate(a) at (0,1);
\coordinate(b) at (0.705,0.705);
\coordinate(c) at (1,0);
\coordinate(d) at (0.705,-0.705);
\coordinate(e) at (0,-1);
\coordinate(f) at (-0.705,-0.705);
\coordinate(g) at (-1,0);
\coordinate(h) at (-0.705,0.705);

\draw (a)--(b);
\draw (b)--(c);
\draw (c)--(d);
\draw (d)--(e);
\draw (e)--(f);
\draw (f)--(g);
\draw (g)--(h);
\draw (h)--(a);
\draw (a)--(e);
\draw (b)--(f);
\draw (c)--(g);
\draw (d)--(h);
\foreach \p in {a,b,c,d,e,f,g,h}
       \fill [black] (\p) circle (3pt);
\end{tikzpicture} \\
$H_8$
\end{tabular}
\begin{tabular}{c}
\begin{tabular}{c}
\begin{tikzpicture}[scale=1]
\coordinate(a) at (0.98480775301,0.17364817766);
\coordinate(b) at (0.64278760968,0.76604444311);
\coordinate(c) at (0,1);
\coordinate(d) at (-0.64278760968,0.76604444311);
\coordinate(e) at (-0.98480775301,0.17364817766);
\coordinate(f) at (-0.86602540378,-0.5);
\coordinate(g) at (-0.34202014332,-0.93969262078);
\coordinate(h) at (0.34202014332,-0.93969262078);
\coordinate(i) at (0.86602540378,-0.5);
\draw (a)--(d);
\draw (a)--(e);
\draw (a)--(f);
\draw (a)--(g);
\draw (b)--(d);
\draw (b)--(e);
\draw (b)--(f);
\draw (b)--(g);
\draw (c)--(f);
\draw (c)--(g);
\draw (c)--(h);
\draw (c)--(i);
\draw (d)--(h);
\draw (d)--(i);
\draw (e)--(h);
\draw (e)--(i);
\draw (f)--(h);
\draw (g)--(i);
\foreach \p in {a,b,c,d,e,f,g,h,i}
       \fill [black] (\p) circle (3pt);
\end{tikzpicture}\\
\end{tabular}\\
$H_9$
\end{tabular}
}
\end{center}
\caption{Graphs $H_8$ and $H_9$.\label{figure:WG}}
\end{figure}
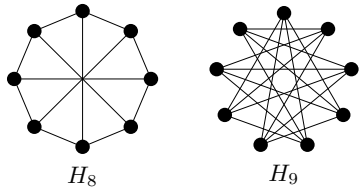

By Lemma~\ref{lemma:quartic}, we are in conditions to apply Theorem~\ref{theorem:summary} choosing $G_0$ as the complement of each of the $5$ graphs in the set $\{H_5,H_6,H_7,H_8,H_9\}$ defined in Lemma~\ref{lemma:quartic}. As a consequence, there exists some positive integer $n_0$ such that each graph $H_n$ defined in Lemma~\ref{lemma:quartic} is the only $t$-optimal for all $n\geq n_0$ thus proving the following result.

\begin{corollary}\label{corollary:example}
There exists a positive integer $n_0$ such that $H_n$ is the only $t$-optimal graph in $\mathcal{R}_{n-5}(n)$ for all $n\geq n_0$.     
\end{corollary}

We remark that Corollary~\ref{corollary:example} is just an example of our methodology. Similar analysis could be conducted to find new $t$-optimal graphs among other classes of trace-minimal graphs. Further research should be carried out to determine the least integer $n_0$ satisfying the conditions of Corollary~\ref{corollary:example}.

\section*{Acknowledgments}
This work is partially supported by City University of New York project entitled \emph{On the problem of characterizing graphs with maximum number of spanning trees}, grant number 66165-00. We acknowledge the anonymous reviewers for their valuable feedback which contributed to the improvement of this manuscript.

\end{document}